\documentclass{amsart}

\usepackage{amsthm}
\usepackage{enumerate}
\usepackage{amsmath}
\usepackage{amssymb}
\usepackage{amsfonts}

\newcommand{\rp}{\mathbb{R}}
\newcommand{\cp}{\mathbb{C}}

\newcommand{\pp}{\mathbb{P}}
\newcommand{\np}{\mathbb{N}}

\newcommand{\zp}{\mathbb{Z}}
\newcommand{\kp}{\mathbb{K}}

\newcommand{\wt}{\widetilde}
\DeclareMathOperator{\coef}{coeff}
\DeclareMathOperator{\Spec}{Spec}
\DeclareMathOperator{\res}{res}

\DeclareMathOperator{\Aut}{Aut}

\DeclareMathOperator{\Card}{Card}
\DeclareMathOperator{\Vol}{Vol}
\DeclareMathOperator{\Div}{div}
\DeclareMathOperator{\Tr}{Tr}

\DeclareMathOperator{\Conv}{Conv}
\DeclareMathOperator{\Divisor}{Div}
\DeclareMathOperator{\interieur}{int}

\DeclareMathOperator{\B}{B}

\DeclareMathOperator{\R}{R}
\newtheorem{lemm}{Lemma}
\newtheorem{coro}{Corollary}
\newtheorem{prop}{Proposition}

\newtheorem{theo}{Theorem}
\newtheorem{rema}{Remark}

\author{Martin Weimann}
\address{Departament Algebra i Geometria, Facultat de Matem\`{a}tiques, Universitat Barcelona
Gran via 585, 08007 Barcelona.}
\email{weimann23@gmail.com}

\title[A lifting and recombination algorithm...]{A lifting and recombination algorithm for rational factorization of sparse polynomials}

\begin{document}


\begin{abstract}
We propose a new lifting and recombination scheme for rational bivariate polynomial factorization that takes advantage of the Newton polytope geometry. We obtain a deterministic algorithm that can be seen as a sparse version of an algorithm of Lecerf, with now a polynomial complexity in the volume of the Newton polytope. We adopt a geometrical point of view, the main tool being derived from some algebraic osculation criterions in toric varieties.
\end{abstract}

\maketitle


\section{Introduction and main results}
This article is devoted to develop an algorithm for factoring a bivariate polynomial $f$ over a number field $\kp$ by taking advantage of the geometry of its Newton polytope. Geometrically, this corresponds to decomposing the curve defined by $f$ in a suitable toric surface $X$. We will thus talk about a toric factorization algorithm. The usual case of dense polynomials corresponds to the classical projective completion $X=\pp^2$ of the complex plane. Our approach is based on algebraic osculation. The central idea is that we can recover the decomposition of the curve $C\subset X$ defined by $f$ from its restriction to a suitable toric Cartier divisor $D$. In a previous work \cite{W:gnus}, we developed a similar method based on vanishing-sums criterions and obtained an exponential complexity toric factorization algorithm. In contrast, we use here a lifting and recombination model based on a vector space basis computation which conduces to a polynomial complexity algorithm.  Our method can be regarded as a toric version of the algorithms  developed by Lecerf \cite{L:gnus}, \cite{L2:gnus} and by Ch\`eze and Lecerf \cite{CL:gnus} for dense polynomials. 
Let us expose our main results.
\vskip2mm
\noindent
{\bf Main results.} Let $\kp$  be a number field and let $f\in\kp[t_1,t_2]$ be a bivariate polynomial. Suppose that $f$ has monomial expansion
$$
f(t)=\sum_{m\in \np^2} c_m t^m,
$$
where $m=(m_1,m_2)$ and $t^m=t_1^{m_1}t_2^{m_2}$. The Newton polytope $N_f$ of $f$ is the convex hull of the exponents $m$ for which $c_m$ is not zero. An exterior facet $F$ of $N_f$ is a one-dimensional face whose primitive inward normal vector has at least one negative coordinate. The associated facet polynomial of $f$ is the univariate polynomial obtained from $f_F=\sum_{m\in F} c_m t^m$ after a suitable monomial change of coordinates. In all of the sequel, we assume the following hypothesis 
\vskip0mm
\begin{eqnarray*}
 & (H_1) & \textit{The polytope $N_f$ contains the elementary simplex of $\rp^2$.}\\
& (H_2) &  \textit{The exterior facet polynomials of $f$ are squarefree.}
\end{eqnarray*}
\vskip2mm
\noindent
We denote by $\omega$ the matrix multiplication exponent. It's well known \cite{GG:gnus} that $2<\omega<2.37$. In all of the sequel, rational factorization means irreducible factorization over $\kp$. Our main result is the following
\vskip2mm
\noindent
\begin{theo} There is a deterministic algorithm that, given $f\in \kp[t_1,t_2]$ which satisfies $(H_1)$ and $(H_2)$, and given the rational factorization of the exterior facet polynomials of $f$, computes the rational factorization of $f$ with 
$\mathcal{O}(\Vol(N_f)^{\omega})$ 
arithmetic operations in $\kp$.
\end{theo}
\vskip2mm
\noindent

In some cases, our complexity improves that of the fastest actual algorithms which would treat $f$ as a dense polynomial.  For instance, if $N_f$ is the convex hull of $\{(0,0),(a,0),(0,b),(n,n)\}$ for some fixed small integers $a, b$, we obtain a complexity $\mathcal{O}(n^{\omega})$  while the fast algorithm of Lecerf \cite{L2:gnus} would have complexity $\mathcal{O}(n^{\omega+1})$.  In any case, the degree sum of the exterior facet polynomials of $f$ is  smaller than the total degree of $f$ so that the unavoidable univariate factorization step is faster using the toric approach. The gain might be considerable : If $a=b=2$ in the previous example, we use two univariate factorizations in degree $2$ instead of one univariate factorization in degree $2n$. 
\vskip2mm
\noindent

Let us explain the main tools for proving Theorem $1$. By the hypothesis $(H_1)$, we can consider a complete \textit{regular} fan $\Sigma$ that refines the normal fan $\Sigma_f$ of $N_f$ and that contains the regular $2$-dimensional cone generated by the canonical basis of $\rp^2$. Such a fan determines a smooth complete toric surface $X=X_{\Sigma}$ and a torus-equivariant embedding of the affine plane $\cp^2=\Spec \cp[t_1,t_2]$ into $X$. The rational factorization of $f$ correspond to the decomposition over $\kp$ of the Zariski closure $C\subset X$ of the affine curve defined by $f$. 
The geometry of $N_f$ is related to the intersection of $C$ with the boundary divisor
$$
\partial X:=X\setminus \cp^2
$$
of the toric completion $X$, and we want to use this information. 
\vskip2mm
\noindent

Let $\Divisor(X)$ be the group of Cartier divisors of $X$. We definitively fix $D\in \Divisor(X)$ effective with support $|D|=|\partial X|$. For convenience, we identify $D$ with the induced subscheme $(|D|,\mathcal{O}_D)$ of $X$ and we denote by $\Divisor(D)$ the group of Cartier divisors of $D$. The inclusion morphism
$
i:D\rightarrow X
$
induces a restriction map $i^*$ on the subgroup of divisors of $X$ who intersects $D$ properly. In particular, we can consider the restriction 
$$
\gamma_C:=i^*(C)\in\Divisor(D)
$$
of $C$ to $D$. The main idea is that for $D$ chosen with sufficiently big multiplicities, we can recover both the rational and the absolute (over $\bar{\kp}$) factorization of $f$ from $\gamma_C$. 
\vskip2mm
\noindent

The irreducible decomposition of $C\cap \partial X$ over $\kp$ is indexed by the set $\mathcal{P}$ of the monic irreducible rational factors of all of the exterior facet polynomials of $f$ and we decompose $\gamma_C$ accordingly as  
$$
\gamma_C=\sum_{P\in\mathcal{P}} \gamma_P,
$$
where $\gamma_P$ corresponds to lifting $P$ to a local factor of $f$ modulo a local equation of $D$ (see Subsection $3.1$). The recombination problem consists in computing the partition of $\mathcal{P}$ that corresponds to the rational decomposition of $C$. To this aim, we introduce the free $\zp$-module
$$
V_{\zp}:=\Big\{\sum_{P\in\mathcal{P}} \mu_P\gamma_P, \,\,\mu_P\in \zp\Big\}\subset \Divisor(D),
$$
and the submodule
$$
V_{\zp}(D):=\{\gamma\in V_{\zp}; \,\,\exists \, E\in \Divisor(X), \,\, i^*(E)=\gamma\}
$$
of divisors of $D$ that extend to $X$. We set $V:=V_{\zp}\otimes\kp$ and $V(D):=V_{\zp}(D)\otimes\kp$. By construction, the irreducible rational decomposition 
$
C=C_1\cup \cdots \cup C_s
$
of $C$ generates a vector subspace 
$$
\langle \gamma_1,\ldots,\gamma_s\rangle \subset V(D)
$$
where $\gamma_j:=i^*(C_j)$. By the hypothesis $(H_2)$, the $\gamma_j$'s are pairwise orthogonal in the basis $(\gamma_P)_{P\in \mathcal{P}}$ of $V$ and so $\dim V(D)\ge s$. The following theorem asserts that equality holds for $D$ big enough. We say that a basis $(\nu_1,\ldots,\nu_n)$ of $V(D)\subset V$ is a \emph{reduced echelon basis} of $V(D)$ if the matrix with $j^{th}$ row $\nu_j$ is in its reduced echelon form in the canonical basis of the input space $V$ (see \cite{Sto:gnus}). Such a basis exists and is unique. We obtain the following
\vskip3mm
\noindent
\begin{theo}
Let $\Div_{\infty}(f)$ be the polar divisor of the rational function of $X$ induced by $f$. If the inequality
$$
D\ge 2\Div_{\infty}(f)
$$
holds, then $(\gamma_1,\ldots,\gamma_s)$ is the reduced echelon basis of $V(D)$. 
\end{theo}
\vskip3mm
\noindent

The proof consists in associating to $\gamma\in V_{\zp}(D)$ a rational $1$-form  with polar divisor controled by $C$. For $D$ big enough, that form is closed and a theorem of Ruppert \cite{Ruppert:gnus} combined with a Galois theory argument permits to conclude that $\gamma$ is $\zp$-combination of the $\gamma_i$'s. There are examples in the dense case that show that the precision $D=2\Div_{\infty}(f)$ in Theorem $2$ is asymptotically sharp  (see \cite{L2:gnus}). 

\vskip2mm
\noindent

In order to apply Theorem $2$ to the factorization problem, we need to determine an explicit system of equations that gives the vector subspace $V(D)\subset V$. To this aim, we use a theorem of the author that characterizes the lifting property. We show in \cite{W:gnus} that there exists a morphism
$$
\Psi:\Divisor(D)\otimes \cp \rightarrow H^0(X,\Omega_{X}^2(D))^{\vee}
$$
so that $\gamma\in \Divisor(D)$ extends to $X$ if and only if $\Psi(\gamma)=0$. Roughly speaking, the linear form $\Psi(\gamma)$ sends a rational form $\omega\in H^0(X,\Omega_X^2(D))$ to the sum of residues of a primitive of $\omega$ along a local analytic lifting curve of $\gamma$. In some sense, this result can be regarded as a converse to the classical residue theorem, we refer to \cite{W:gnus} for details. This permits to prove the following 
\vskip2mm
\noindent
\begin{theo}
Suppose that $D$ satisfies the hypothesis of Theorem $2$. Then, 
$$
V(D) =\ker (A)
$$
for some explicit matrix $A=(a_{P,m})_{P\in \mathcal{P},m\in M}$ with coefficients in $\kp$, where $M$ is the set of interior lattice points of the polytope $2N_f$. 
\end{theo}
\vskip2mm
\noindent

So we can solve the recombination problem with linear algebra over $\kp$. Then, we compute the rational factors of $f$ by solving systems of affine equations. We finally obtain a deterministic polynomial complexity algorithm for rational toric factorization of bivariate polynomials. We describe briefly the main steps of the algorithm. The given complexities are obtained in Corollaries $1$ and $2$ in Section $3$. As in \cite{GG:gnus}, we use the notation $\wt{\mathcal{O}}$ for the soft complexity.   
\vskip4mm
\noindent
{\bf Toric Factorization Algorithm}
\vskip2mm
\noindent
{\bf Input:} $f\in \kp[t_1,t_2]$  satisfying hypothesis $(H_1)$ and $(H_2)$.
\vskip1mm
\noindent
{\bf Output:} The irreducible rational factors of $f$.
\vskip2mm
\noindent
{\it Step $0$: Univariate factorization.} Compute the set $\mathcal{P}$ of irreducible rational factors of the exterior facet polynomials of $f$. 
\vskip2mm
\noindent
{\it Step $1$: Lifting.} This is the $\gamma_P$'s computation step. For each $P\in\mathcal{P}$, compute the associated local factor of $f$ modulo  the local equation of $D=2\Div_{\infty}(f)$. This step has complexity $\wt{\mathcal{O}}(\Vol(N_f)^{2})$.
\vskip2mm
\noindent
{\it Step $2$: Recombination.}  

a) Build the matrix $A$ of Theorem $2$. This step has complexity $\wt{\mathcal{O}}(\Vol(N_f)^{2})$.

b) Compute the reduced echelon basis associated to $A$. This step has complexity $\mathcal{O}(\Vol(N_f) \Card(\mathcal{P})^{\omega-1})$.
\vskip2mm
\noindent
{\it Step $3$: Factors computation.} Solve some affine systems of linear equations over $\kp$ to recover the rational factors of $f$. This step has complexity $\mathcal{O}(\Vol(N_f)^{\omega})$.
\vskip4mm
\noindent

As already mentioned, the great advantage of our algorithm is that it replaces the usual univariate factorization in degree $d=\deg(f)$ by the factorization of the exterior facet polynomials of $f$ : their degree sum is at most $d$, and much smaller in many significant cases. It follows that both the number of unknowns and equations in the recombination process decrease too and the basis computation step $2$ $b)$ is faster than in \cite{L2:gnus}. Steps $1$ and $2$ $a)$ rely on classical modular algorithms (Newton iteration, modular multiplication) whose complexity analysis is delicate due to the sparseness of $f$. This partially explains that our lifting complexity does not reach the soft complexity $\wt{\mathcal{O}}(d^{2})$ obtained in \cite{L2:gnus} for dense polynomials. Step $3$ has the highest cost of the algorithm because in the general sparse case we have to use linear algebra instead of the fast multiplication or partial fraction decomposition methods that are used for dense polynomials (\cite{Gao1:gnus}, \cite{L:gnus}, \cite{CL:gnus}). Finally, let us mention that the  algorithm developed by Lecerf \cite{L:gnus} in the bidegree case suggests that it is possible to reduce both the number of facet factorizations and the lifting precision. We refer to Subsection $3.5$ and Section $4$ for further comments.

\vskip3mm
\noindent
{\bf Related results.} Classical results about polynomial factorization can be found in \cite{GG:gnus}. For more recent advances, we refer the reader to the introduction of \cite{CL:gnus} (and to the complete list of references therein) that gives a large and comprehensive overview of the current algorithms for factorization of polynomials. We only discuss here the most related results. 

\vskip1mm
\noindent {\it Using linear algebra.} Factoring multivariate polynomials by means of linear algebra has been made possible by the powerful irreducibility criterion of Ruppert \cite{Ruppert:gnus}. This is the so-called logarithmic derivative method, that relates the basis computation of the vector space of closed rational $1$-forms with some appropriate polar divisor. This point of view has been developed by Gao in \cite{Gao1:gnus}, who combined the logarithmic derivative method with the Rothstein-Trager algorithm for absolute partial fraction decomposition (\cite{GG:gnus}, Theorem $2.8$). Finally, as pointed out in the introduction, Lecerf \cite{L2:gnus}, \cite{L:gnus} and Ch\`eze-Lecerf \cite{CL:gnus} recently developed very efficient hybrid algorithms for rational and absolute factorization, by combining Gao's approach with a lifting and recombination scheme. This is the point of view we follow here. 

\vskip1mm
\noindent {\it Using Newton polytopes.} Factoring polynomials by taking into account the Newton polytope is an active area of research. In \cite{EGW:gnus}, M. Elkadi, A. Galligo and the author use some probabilistic interpolation criterions \cite{W1:gnus}, by replacing the divisor $D$ with a generic ample curve ``close to the boundary''. In \cite{W:gnus}, the author looks for the \textit{effective} decompositions of $\gamma_C$ that may be lifted to $X$. By taking into account natural degree conditions imposed by the Minkowski-sums decompositions of $N_f$, there appears supplementary vanishing cohomology properties of the osculating divisors that permit to use the smaller precision $D=\Div_{\infty}(f)+\partial X$ (see also \cite{L:gnus} for a similar comparison in the dense case). In return, it gives a problem of partitions of $V(D)\cap \{0,1\}^{\mathcal{P}}$ that has exponential complexity in the worst case (see Subsection $3.6$). A comparable algorithm is obtained in  \cite{Gao:gnus}, where the authors use a more combinatorial approach. In \cite{Sombra:gnus}, the authors show that the low degree factors of $f$ can be  computed in polynomial time with respect to the fewnomial encoding of $f$.

\vskip3mm
\noindent
{\bf Organization.} We prove Theorem $2$ in the next Section $2$. In Section $3$, we develop a toric factorization algorithm and we prove Theorem $1$. In Section $4$, we compare our method with the most related dense and toric algorithms and we discuss some possible improvements.  We conclude in Section $5$.

\section{Proof of Theorem $2$}

We follow the  notations of the introduction. 
We saw that the rational decomposition 
$C=C_1\cup \cdots \cup C_s$ of $C$ generates a vector subspace $\langle \gamma_1,\ldots,\gamma_s\rangle\subset V(D)$ and we want to show
that the opposite inclusion holds when $D\ge 2\Div_{\infty}(f)$. The strategy consists in associating to $\gamma\in V_{\zp}(D)$ a closed rational $1$-form $\omega$ on $X$ whose polar divisor is controled by $C$. A theorem of Ruppert \cite{Ruppert:gnus} implies that $\omega$ is a $\cp$-linear combination of the logarithmic derivatives of the absolute factors of $f$. Finally, we conclude by Galois theory that $\gamma$ is $\zp$-combination of the $\gamma_j$'s. 

We need first two preliminaries lemmas that clear up the behaviour of restriction  with respect to derivation. The remaining part of the proof will follow in Subsection $2.2$.  If not specified, all schemes are considered over $\cp$.

\subsection{Notations and preliminaries lemmas.}

We denote by $I_D$ the ideal sheaf of $D$ and by $\mathcal{O}_D$ its structural sheaf. The structural sequence of $D$ is
\begin{eqnarray}
0\rightarrow I_D\rightarrow \mathcal{O}_X \stackrel{i^*}{\rightarrow} \mathcal{O}_D\rightarrow 0,
\end{eqnarray}
where the restriction map $i^*$ is induced by the inclusion $i:D\rightarrow X$.
We denote by $\mathcal{O}_X(D)$ the sheaf of rational functions with polar divisor bounded by $D$, by $\Omega_X^q$ the sheaf of regular $q$-forms and we let $\Omega_X^q(D):=\Omega_X^q\otimes \mathcal{O}_X(D)$.

We say that $B\in \Divisor(X)$ is a normal crossing divisor if it has local equation $x_1\cdots x_r=0$ where the $x_i$'s form part of a local system of coordinates $(x_1,\ldots,x_n)$ of $X$ (so $n=2$ in our case). For such a $B$, we introduce the sheaf $\Omega_X^q(\log B)$ of rational $q$-forms with logarithmic poles along $B$. By definition, $\phi\in \Omega_X^q(\log B)$ if and only if both $h\phi$ and $h d\phi$ are regular for some local equation $h=0$ of $B$. It is well known that $\Omega_X^q(\log B)$ is a locally free sheaf of $\mathcal{O}_X$-module \cite{Voisin:gnus}.

\vskip1mm
\noindent

The following lemma clears up the behaviour of the restriction morphism with derivation.

\vskip1mm
\noindent
\begin{lemm} Let $B$, $D$ as before, with $|D|\subset |B|$. Let $F$ be an effective divisor which intersects $D$ properly. The differential $d$ induces a commutative diagram
$$
\begin{array}{ccccccccc}
\Omega_X^1(\log B)\otimes \mathcal{O}_X(F) & \stackrel{d}{\rightarrow} & \Omega_X^2(B)\otimes \mathcal{O}_X(2F) \\
\downarrow i^*  &  & \downarrow   i^*    \\
\Omega_X^1(\log B)\otimes \mathcal{O}_D(F) & \stackrel{d_D}{\rightarrow} & \Omega_X^2(B)\otimes \mathcal{O}_D(2F). \\
\end{array}
$$
\end{lemm}
\vskip1mm
\noindent
\begin{proof}
We show Lemma $1$ for an arbitrary smooth complete variety $X$ of dimension $n$. Since $B$ is normal crossing, it has local equation $x_1\cdots x_r=0$ where the $x_i$'s form part of a local system of coordinates $(x_1,\ldots,x_n)$ of $X$. The sheaf $\Omega_X^1(\log B)$ is a locally free  sheaf of $\mathcal{O}_X$-modules and a germ $\phi\in \Omega_X^1(\log B)\otimes \mathcal{O}_X(F)$ has a unique representation
$$
\phi=h_1 dx_1/x_1+\cdots +h_r dx_r/x_r+h_{r+1}dx_{r+1}+\cdots+h_ndx_n
$$
for some $h_i\in \mathcal{O}_X(F)$ (see \cite{Voisin:gnus}, p. 186). It is clear that $dh_i$ has its polar divisor bounded by $2F$. We deduce that 
$$
d\phi=dh_1\land dx_1/x_1+\cdots +\cdots+dh_n\land dx_n
$$
belongs to $\Omega_X^2(B+2F)$ and the upper row is well-defined. In order to show that $d_D$ is well-defined, we need to show that
$$
\phi\in  \Omega_X^1(B+F)\otimes I_D\Longrightarrow d\phi\in \Omega_X^2(B+2F)\otimes I_D.
$$
Since $D$ is supported on $|B|$, it has local equation $x^k:=x_1^{k_1}\cdots x_r^{k_r}$ for some $k_i\in \np$. Thus if $\phi\in  \Omega_X^1(B+F)\otimes I_D$, we have $h_i=x^k h'_i$ for some $h_i'\in \mathcal{O}_X(F)$ and 
$$
\frac{dh_i}{x^k} \land \frac{dx_i}{x_i} =h'_i \sum_{j=1}^r k_j \frac{dx_j}{x_j}\land \frac{dx_i}{x_i}+ dh'_i\land \frac{dx_i}{x_i}
$$
belongs to $\Omega_X^2(B+2F)$ for all $i=1,\ldots,r$. In the same way, it is easy to check that $dh_i\land dx_i/x^k\in \Omega_X^2(B+2F)$ for $i>r$. Multiplying by $x^k$, we obtain that  $d\phi\in \Omega_X^2(B+2F)\otimes I_D$. 
\end{proof}

We now pay attention to the behaviour of the restriction map with logarithmic derivation. 
We denote by $\mathcal{M}_{X,D}$ the sheaf of rational functions whose polar locus intersects $D$ properly. We have an exact sequence
\begin{eqnarray}
0\rightarrow I_D\mathcal{M}_{X,D}\rightarrow \mathcal{M}_{X,D}\stackrel{i^*}{\rightarrow}\mathcal{M}_{D}\rightarrow 0,
\end{eqnarray}
where $\mathcal{M}_D:=\mathcal{M}_{X,D}\otimes\mathcal{O}_D$ is the sheaf of rational sections of $\mathcal{O}_D$. The multiplicative version of $(2)$ is
\begin{eqnarray}
0\rightarrow 1+I_D\mathcal{M}_{X,D}\rightarrow \mathcal{M}_{X,D}^*\stackrel{i^*}{\rightarrow}\mathcal{M}_{D}^*\rightarrow 0,
\end{eqnarray}
where $^*$ stands for the multiplicative sheaves of units. On the other hand, the logarithmic derivative $d\log (h):=dh/h$ induces the natural morphisms
\begin{eqnarray}
d\log : \mathcal{M}_{X,D}^*\rightarrow \Omega_X^1\otimes \mathcal{M}_{X,D}
\quad {\rm and} \quad d\log : \mathcal{O}_{X}^*\rightarrow \Omega_X^1
\end{eqnarray}
of sheaves of abelian groups. We have the following
\vskip1mm
\noindent
\begin{lemm} 
Let $B$, $D$ as before and suppose that $|D|\subset |B|$. The  morphisms in $(4)$ combined with the natural inclusion $j:\Omega_X^1\otimes \mathcal{M}_{X,D}\rightarrow \Omega_X^1(\log B)\otimes \mathcal{M}_{X,D}$ induce the commutative diagram
$$
\begin{array}{ccccccccc}
 \mathcal{M}_{X,D}^* \qquad\qquad & \stackrel{i^*}{\rightarrow} & \mathcal{M}_{D}^* \qquad & \rightarrow & 0\\
  \downarrow j \circ d\log & & \downarrow d_D\log &  \\
 \Omega_X^1(\log B)\otimes \mathcal{M}_{X,D}& \stackrel{i^*}{\rightarrow} & \Omega_X^1(\log B)\otimes \mathcal{M}_{D}& \rightarrow & 0,\\
\end{array}
$$
and its regular version
$$
\begin{array}{ccccccccc}
 \mathcal{O}_{X}^* & \stackrel{i^*}{\rightarrow} & \mathcal{O}_{D}^* & \rightarrow & 0\\
   \downarrow  & & \downarrow \\
 \Omega_X^1(\log B) & \stackrel{i^*}{\rightarrow} & \Omega_X^1(\log B)\otimes \mathcal{O}_D & \rightarrow & 0.\\
\end{array}
$$
\end{lemm}
\vskip1mm
\noindent

\begin{proof}
Let $\alpha\in \mathcal{M}_{D}^*$ and $u\in\mathcal{M}_{X,D}^*$ so that $\alpha=i^*(u)$. By $(4)$, the morphism 
$$
d_D\log : \mathcal{M}_{D}^* \rightarrow \Omega^1_X(\log B)\otimes\mathcal{M}_{D},\quad d_D\log(\alpha):=i^*(j\circ d\log (u))
$$
will be well-defined if we show that 
$$
u\in 1+I_D\mathcal{M}_{X,D}\Rightarrow d_D\log(u)=0.
$$
So let $u=1+h$, for $h\in I_{D}\mathcal{M}_{X,D}$ a germ at some smooth point of $B$. Thus $B$ and $D$ have respective local equation $x=0$ and $x^k=0$ for some $k\ge 0$, and $h=mx^{k}$ where $m\in \mathcal{M}_{X,D}$. So
$$
d\log(1+h)=\frac{kmx^{k-1}dx+x^kdm}{1+x^k m}=\frac{mx^{k}}{1+x^km}\big(k\frac{dx}{x}+\frac{dm}{m}\big)
$$
belongs to the subsheaf $\Omega_{X}^1(\log(B))\otimes I_D\mathcal{M}_{X,D}\subset \Omega_X^1\otimes \mathcal{M}_{X,D}$. By tensoring $(2)$ with the locally free sheaf $\Omega_{X}^1(\log(B))$ we deduce that 
$$
i^*(j\circ d\log(1+h))=0\,\,\in \Omega^1_X(\log B)\otimes\mathcal{M}_{D}.
$$
The divisor $B$ being normal crossing, we check easily that the same conclusion holds when $h$ is a germ at some singular point of $B$. This implies that the map $d_D\log$ is well-defined, giving the first diagram. The regular version follows by letting $m\in \mathcal{O}_X$ in the previous reasonning and by using the multiplicative version of $(1)$.
\end{proof}
\vskip1mm
\noindent

\subsection{Proof of Theorem $2$.} 

We come now to the proof of Theorem $2$. We denote by 
$$
\mathbb{T}:=\Spec \cp[t_1^{\pm 1},t_2^{\pm 1}]\quad {\rm and}\quad \cp^2:=\Spec \cp[t_1,t_2]
$$
the complex torus of $X$ and the affine plane endowed with canonical coordinates $t=(t_1,t_2)$. We identify rational forms of $\mathbb{T}$ and $\cp^2$ with the rational form they induce on $X$.  We suppose from now that $D=2\Div_{\infty}(f)$ and that $B=X\setminus \mathbb{T}$. The toric surface $X$ being smooth, the toric divisor $B$ is normal crossing. Since $f$ has no poles in the torus, we have $|D|\subset|B|$.

Let $\gamma\in V(D)$. We need to show that $\gamma$ is linear combinations of the $\gamma_j$'s. There is no loss of generality to suppose that $\gamma\in V_{\zp}(D)$. Let $\Gamma:=C\cap \partial X$ be the intersection of $C$ with the boundary. By hypothesis, we know that
\begin{eqnarray}
\gamma=\sum_{p\in\Gamma} \mu_p\gamma_p\quad{\rm and}\quad \gamma=i^*(C_{\gamma}),
\end{eqnarray}
where $\gamma_p\in \Divisor(D)$ is induced by the germ of $C$ at $p$, $\mu_p$ is an integer and where $C_{\gamma}\in \Divisor(X)$. Since the torus has a trivial Chow goup, there exists $E_{\gamma}$ supported on $B$ so that 
$$
\Div(g)=C_{\gamma}-E_{\gamma}
$$ 
for some rational function $g\in \cp(X)$. The  following key lemma ensures that the poles of the restriction $i^*(dg/g)$ are controled by $C$. 
\vskip1mm
\noindent

\begin{lemm} We have
$
i^*(dg/g)\in H^0(X,\Omega_X^1(\log B)\otimes \mathcal{O}_D(C)).
$
\end{lemm}
\vskip1mm
\noindent

\begin{proof}
Obviously, $i^*(dg/g)$ defines  a  rational section of $\Omega_{X}^1(\log B)\otimes \mathcal{O}_{D}$ and we need to show that the germ $g_p$ of $g$ at $p$  satisfies
$$
i^*(dg_p/g_p)\in \Omega_{X}^1(\log B)\otimes \mathcal{O}_{D}(C)
$$
for all $p\in |D|$ (for convenience, we omit the index $p$ in the stalk notations). Let us write $g_p=G_p/H_p$ for some local equations $G_p$ and $H_p$ of respectively $C_{\gamma}$ and $E_{\gamma}$. Thus
\begin{eqnarray}
dg_p/g_p=dG_p/G_p-dH_p/H_p.
\end{eqnarray}
Since $E_{\gamma}$ is supported on $B$, we have $dH_p/H_p\in \Omega_{X}^1(\log B)$ and it's enough to show that 
$$
i^*(dG_p/G_p)\in \Omega_{X}^1(\log B)\otimes \mathcal{O}_{D}(C)
$$
for all $p\in|D|$. 
For convenience, we let $\gamma_p:=0$ and $\mu_p:=0$ for $p\in |D|\setminus |\Gamma|$.  By $(5)$, the germ $(C_{\gamma},p)$ of $C_{\gamma}$ at $p$ satisfies
$$
i^*(C_{\gamma},p)=\mu_p\gamma_p=\mu_p i^*(C,p)
$$
for all $p\in|D|$. This is equivalent to that
\begin{eqnarray}
i^*(G_p/F_p^{\mu_p})\in \mathcal{O}_{D}^*,
\end{eqnarray}
where $F_p$ is any local equation of  $C$ at $p$. Lemma $2$ combined with $(7)$ implies that
\begin{eqnarray}
\qquad\quad d_D\log(i^*(G_p/F_p^{\mu_p}))=i^*(dG_p/G_p)-\mu_p i^*(dF_p/F_p)\in\Omega_{X}^1(\log B)\otimes \mathcal{O}_D.
\end{eqnarray}
Since $i^*(dF_p/F_p)\in \Omega_{X}^1(\log B)\otimes \mathcal{O}_{D}(C)$, 
it follows that $i^*(dG_p/G_p)$ belongs to $\Omega_{X}^1(\log B)\otimes \mathcal{O}_{D}(C)$.
\end{proof}
\vskip1mm
\noindent

\begin{lemm} 
There exists a unique rational form $\omega\in H^0(X,\Omega_X^1(\log(B)\otimes \mathcal{O}_X(C))$ such that $i^*(dg/g)=i^*(\omega)$.
\end{lemm}
\vskip1mm
\noindent

\begin{proof}
By tensoring $(2)$ with the locally free sheaf $\Omega_X^1(\log B)\otimes \mathcal{O}_X(C)$ and by looking at the associated long exact cohomological sequence, we deduce from Lemma $3$ that it is enough to show that
\begin{eqnarray}
H^1(X,\Omega_X^1(\log B)\otimes \mathcal{O}_X(C)\otimes I_D)=0.
\end{eqnarray}
But $C-\Div_{\infty}(f)=\Div (f)$ being principal, multiplication by $f^2$ gives a global isomorphism 
$$
\mathcal{O}_X(C)\otimes I_D\simeq \mathcal{O}_X(-C).
$$
Moreover, we know by \cite{F:gnus} that the sheaf $\Omega_X^1(\log B)$ is globally trivial. Thus, there is a global isomorphism
$$
\Omega_X^1(\log B)\otimes \mathcal{O}_X(C)\otimes I_D\simeq \mathcal{O}_X(-C)\oplus \mathcal{O}_X(-C).
$$
Since $C$ has a non negative intersection with each irreducible toric divisor of $X$, the line bundle $\mathcal{O}_X(C)$ is numerically effective \cite{F:gnus}. Moreover, it is well known \cite{F:gnus} that 
we have equality
\begin{eqnarray}
H^0(X,\mathcal{O}_X(C))= \{t^m/f,\,\,m\in N_f\cap \zp^2\}.
\end{eqnarray}
By $(H_1)$, the Newton polytope $N_f$ has dimension $2$, and we deduce from $(10)$ that $\mathcal{O}_X(C)$ is big. It is a standard result that
$$
H^1(X,\mathcal{O}_X(-C))=0
$$
for any big and nef Cartier divisor $C$ on a smooth complete variety $X$ (see \cite{Laz:gnus}, Theorem $4.5$ for instance). 
Finally, $(9)$ holds and there exists $\omega$ as in Lemma $4$. Since $H^0(X,\mathcal{O}_X(-C))=0$, the previous reasonning implies that such an $\omega$ is unique. 
\end{proof}

Up to here, we can show that all previous lemmas would remain valid with the choice $D=\Div_{\infty}(f)+\partial X$ used in \cite{W:gnus}. The choice $D\ge 2\Div_{\infty}(f)$ appears to be essential in order to have
\vskip1mm
\noindent
\begin{lemm}
We have $d\omega=0$.
\end{lemm}
\vskip1mm
\noindent
\begin{proof}
Lemma $1$ applied with $F=C$ combined with Lemma $4$ gives
$$
i^*(d\omega)=d_D(i^*(\omega)) = d_D(i^*(dg/g))=i^*(d(dg/g))=0,
$$
where $i^*(d\omega)\in H^0(X,\Omega_X^2(B+2C)\otimes \mathcal{O}_D)$. By tensoring $(2)$ with $\Omega_X^2(B+2C)$,  and by using the associated long exact cohomological sequence, we deduce that
$$
d\omega \in H^0(X,\Omega_X^2(B+2C)\otimes I_D).
$$
Since $2C-D=\Div(f^2)$, it follows that
$
f^2 d \omega\in H^0(X,\Omega_X^2(B)).
$
By \cite{F:gnus}, the divisor $B=X\setminus \mathbb{T}$ is an anticanonical divisor of $X$ and there is an identification 
$$
H^0(X,\Omega_X^2(B))=\cp \frac{dt_1\land dt_2}{t_1 \,\,t_2}.
$$
Thus,
$$
d \omega=\frac{c}{f^2}\frac{dt_1\land dt_2}{t_1\,t_2}
$$
for some constant $c\in\cp$. This form being exact, its residue at zero vanishes. This forces
$
c/f^2(0)=res_0(d\omega)=0
$
(recall that $f(0)\ne 0$ by $(H_1)$) and $d\omega=0$.
\end{proof}

Here comes a theorem of Ruppert in the picture. 
\vskip1mm
\noindent
\begin{lemm}
There exists some constants $c_j, a_1, a_2\in\cp$ such that
$$
\omega=\sum_{j=1}^t c_j \frac{d\bar{q}_j}{\bar{q}_j}+a_1\frac{dt_1}{t_1}+a_2\frac{dt_2}{t_2},
$$
with $\bar{q}_1,\cdots, \bar{q}_t$ the irreducible absolute factors of $f$. 
\end{lemm}
\vskip1mm
\noindent
\begin{proof}
By \cite{F:gnus}, the map $m\mapsto dt^m/t^m$ gives an isomorphism
$
\zp^2\otimes \mathcal{O}_X \simeq\Omega_X^1(\log B).
$
We deduce that the map
\begin{eqnarray*}
H^0(X,\mathcal{O}_X(C))\oplus H^0(X,\mathcal{O}_X(C)) & \rightarrow & H^0(X,\Omega_X^1(\log B)\otimes \mathcal{O}_X(C))\\
(r_1,r_2)\quad & \mapsto & r_1 dt_1/t_1+r_2 dt_2/t_2
\end{eqnarray*}
is an isomorphism. It follows from $(10)$ that there exists (unique) polynomials $h_1,h_2$ such that
\begin{eqnarray}
\omega=\frac{h_1}{f}\frac{dt_1}{t_1}+\frac{h_2}{f}\frac{dt_2}{t_2},\quad N_{h_i}\subset N_f.
\end{eqnarray}
On the other hand, $\omega$ being closed by Lemma $5$, its restriction to $\cp^2$ defines an element of the first algebraic De Rham cohomology group $H^1(\cp^2\setminus C_0)$, where 
$$
C_0:=(C+B)\cap \cp^2=\{t_1t_2 f=0\}.
$$
By a theorem of Ruppert \cite{Ruppert:gnus}, it follows that there are uniquely determined constants $c_1,\ldots,c_t, a_1, a_2\in\cp$ and a unique \textit{exact} rational $1$-form $\omega'$ such that
$$
\omega-\omega'=\sum_{i=1}^t c_j \frac{d\bar{q}_j}{\bar{q}_j}+a_1\frac{dt_1}{t_1}+a_2\frac{dt_2}{t_2}.
$$ 
Since the right hand side  can be written as in $(11)$, we deduce that there are polynomials $p_1, p_2$ so that
$$
\omega'=\frac{p_1dt_1+p_2dt_2}{t_1t_2f},\quad \deg (p_i) < \deg (ft_1t_2),
$$
where $\deg(\cdot)$ stands for the total degree. Since the polynomial $t_1t_2f$ is reduced (by $(H_1)$ and $(H_2)$) and $\omega'$ is exact, it follows from \cite{BC:gnus}, Proposition 3 that $\omega'=0$. This gives the desired expression for $\omega$.
\end{proof}
\vskip1mm
\noindent

\begin{lemm}
Let $\bar{C}_j\in \Divisor(X)$ be the component of $C$ defined by the absolute factor $\bar{q}_j$. 
We have the relation $\gamma=\sum_{j=1}^t c_j i^*(\bar{C}_j)$. 
\end{lemm}

\vskip1mm
\noindent

\begin{proof}
Let $p\in|\Gamma|$ and let $x=0$ and $y=0$ be some respective local equations of $C$ and $B$ at $p$. By Lemma $6$, the germ $\omega_p$ of $\omega$ at $p$ satisfies
$\omega_p-c_j\,dy/y\in\Omega_{X}^1(\log B)$, where $\bar{C}_j$ is the unique component of $C$ passing throw $p$ (unicity of $\bar{C}_j$ comes from $(H_2)$). It follows that
$$
i^*(\omega_p)-c_j i^*(dy/y)\in\Omega_{X}^1(\log B)\otimes \mathcal{O}_D,
$$
while $(6)$ combined with $(8)$ implies that
$$
i^*(dg_p/g_p)-\mu_p i^*(dy/y)\in \Omega_{X}^1(\log B)\otimes \mathcal{O}_D.
$$
Using equality $i^*(dg_p/g_p)=i^*(\omega_p)$ induced by Lemma $4$, we deduce that
$$
(\mu_p-c_j)i^*(dy/y)\in \Omega_{X}^1(\log B)\otimes \mathcal{O}_D,
$$
so that
$$
(\mu_p-c_j)i^*(dx\land dy/xy)\in \Omega_{X}^2(B)\otimes \mathcal{O}_D.
$$
Hence, there exists $\psi\in \Omega_{X}^2(B)$ such that 
$$
(\mu_p-c_j)dx\land dy/xy-\psi \in \Omega_{X}^2(B)\otimes \mathcal{O}_X(C)\otimes I_D.
$$
Since $|\Div_{\infty}(f)|=|\partial X|$ (see Subsection $3.1$), we deduce that $(D,p)\ge (B,p)$ for any $p\in |\Gamma|$. It follows that the previous germ of $2$-form has its polar divisor bounded by the smooth germ of curve $(C,p)=\{y=0\}$. Hence it has no residue at $p$
$$
\res_p\big[(\mu_p-c_j)dx\land dy/xy-\psi\big]=0.
$$
In the same way, $\psi\in \Omega_{X}^2(B)$  forces $\res_p(\psi)=0$ so that
$$
0=\res_p\big[(\mu_p-c_j)dx\land dy/xy\big]=\mu_p-c_j.
$$
The relation  $\gamma=\sum_{j=1}^t c_ji^*(\bar{C}_j)$ follows. 
\end{proof}

\vskip0mm
\noindent
\begin{lemm}
$\gamma$ is $\zp$-combination of the $\gamma_j$'s. 
\end{lemm}
\vskip0mm
\noindent
\begin{proof}
Let $\bar{C}_j$ and $\bar{C}_k$ be conjugate components. We need to show that $c_j=c_k$. Let us consider the schemes $X$, $D$ , $\bar{C}_j$ and $\gamma$ as schemes over $\Spec \bar{\kp}$ (we keep the same notations for simplicity). Since both $D$ and $\gamma$ are defined over $\kp$, the group  $\Aut_{\kp}(X)$ of $\kp$-automorphisms of $X$ acts on $\Divisor(D)$ and fix $\gamma$.  Let $\sigma\in \Aut_{\kp}(X)$ be such that $\sigma(\bar{C}_j)=\bar{C}_k$. Then 
\begin{eqnarray}
c_k i^*(\bar{C}_k)+\sum_{i\ne k}c_i i^*(\bar{C}_i)=\gamma= \sigma(\gamma)=c_j i^*(\bar{C}_k)+\sum_{i\ne j}c_i \sigma(i^*(\bar{C}_i)).
\end{eqnarray}
Since $\sigma$ induces a permutation of the irreducible absolute components of the rational curve $C$, $(12)$ implies that $(c_j-c_k)i^*(\bar{C}_k)$ is supported on $\sum_{i\ne k}c_i i^*(\bar{C}_i)$. Since by the hypothesis $(H_2)$, the schemes $i^*(\bar{C}_k)$ and $\sum_{i\ne k}i^*(\bar{C}_i)$ have disjoint support, this forces equality $c_j=c_k$.
\end{proof}

This shows that $V(D)=\langle \gamma_1,\ldots,\gamma_s\rangle$. Since the $\gamma_j$'s have  coordinates in $\{0,1\}$ and are pairwise orthogonal in the canonical basis $(\gamma_P)_{P\in\mathcal{P}}$ of the ambient space $V$, they form (under some unique permutation) the reduced echelon basis of $V(D)$. Obviously, this remains true for any choice $D\ge 2\Div_{\infty}(f)$. Theorem $2$ is proved.
$\hfill{\square}$

\vskip3mm
\noindent

\begin{rema}
If we rather consider the $\zp$-module $\bar{V}_{\zp}$ induced by the irreducible decomposition of $\gamma_C=C\cap D$ over the algebraic closure $\bar{\kp}$ of $\kp$, we can check that Theorem $2$ remains valid. Namely, if $D\ge 2\Div_{\infty}(f)$, then the submodule $\bar{V}_{\zp}(D)\subset \bar{V}_{\zp}$ of divisors that extend to $X$ is free generated by the restrictions of the irreducible absolute components of $C$. This might be usefull to compute the absolute factorization of $f$.
\end{rema}

\section{A toric factorization algorithm} 

We come now to the proof of Theorem $1$. We introduce notations and review basic facts about toric geometry in Subsection $3.1$. We describe the echelon basis computation in Subsection $3.2$ and the factors computation in Subsection $3.3$. We develop a toric factorization algorithm and prove Theorem $1$ in Subsection $3.4$. 

\subsection{Notations and preliminaries.} We refer to \cite{Danilov:gnus} and \cite{F:gnus} for an introduction to toric geometry. As before, $X$ designs the smooth toric surface associated to a regular fan $\Sigma$ that refines the normal fan $\Sigma_f$ of $N_f$. 

We denote by $D_0,\ldots,D_{r+1}$ the irreducible toric divisors of $X$ and by $\rho_0,\ldots,\rho_{r+1}$ the corresponding rays of $\Sigma$. Since $\Sigma$ is regular, we can order the $D_i$'s in such a way that the generators $\eta_i$ of the monoids $\rho_i\cap \zp^2$ satisfy 
$
\det(\eta_{i},\eta_{i+1})=1,
$
with the convention $\eta_{r+2}=\eta_{0}$. 
We denote by $U_i\simeq \cp^2$
the affine toric chart associated to the two-dimensional cone $\rho_{i}\rp^+\oplus \rho_{i+1}\rp^+$. Thus, 
$$
U_i= \Spec\,\cp[x,y]
$$
where affine coordinates and torus coordinates $t=(t_1,t_2)$ are related by relations
$$
t^m=x^{\langle m,\eta_i \rangle} y^{\langle m,\eta_{i+1}\rangle}
$$
for all $m\in\zp^2$, with $\langle\cdot,\cdot\rangle$ the standard scalar product (see \cite{Danilov:gnus} for instance). Moreover, the irreducible toric divisors have affine equations 
$$
D_{i}\cap U_i=\{x=0\}\quad{\rm and}\quad D_{i+1}\cap U_i=\{y=0\},
$$ 
and $|D_j|\cap U_i=\emptyset$ for $j\notin\{ i,i+1\}$. 

\vskip2mm
\noindent

Since the fan $\Sigma$ contains the cone generated by the canonical basis $(e_1,e_2)$ of $\rp^2$, we can chose the indexation such that 
$
(\eta_{r+1},\eta_0)=(e_1,e_2)
$, in which case
$$
\partial X=D_1+\cdots+D_r.
$$
It is well known that the toric divisor $D_i$ appears in the principal divisor $\Div(f)$ with multiplicity 
$$
d_i:=-min_{m\in N_f} \langle m,\eta_i\rangle.
$$ 
By $(H_1)$, we  have $d_0=d_{r+1}=0$, and $d_i>0$ otherwise, so that
$$
\Div_{\infty}(f)=d_1D_1\cdots+d_rD_r
$$
for some $d_1,\ldots,d_r \ge 1$. In particular, $|\Div_{\infty}(f)|=|\partial X|$ as asserted in the proof of Lemma $7$.
Since $f=\sum_{m\in N_f \cap \zp^2} c_m t^m$, the rational function 
\begin{eqnarray}
f_i(x,y):=\sum_{m\in N_f \cap \zp^2} c_m x^{\langle m,\eta_i\rangle+d_i}y^{\langle m,\eta_{i+1}\rangle+d_{i+1}}
\end{eqnarray}
is a polynomial which does not vanish at $(0,0)$, and $C$ has local  equation
$$
C\cap U_i =\{f_i(x,y)=0\}
$$
in the chart $U_i$. We call the univariate polynomial $P_i(y):=f_i(0,y)$ the $i^{th}$-facet polynomial of $f$, or exterior facet polynomial of $f$ when $i\ne 0,r+1$. The polynomial $P_i$ has degree
$$
l_i:=\deg(P_i)=\Card(N_f^{(i)}\cap\zp^2)-1,
$$
the lattice length of the $i^{th}$ exterior face
$$
N_f^{(i)}:=\{m\in N_f,\,\langle m,\eta_{i}\rangle=-d_i\}
$$
of $N_f$. If $\rho_i$ belongs to the original normal fan $\Sigma_f$ of $N_f$, this face has dimension one. Otherwise, it is a vertex of $N_f$.
\vskip2mm
\noindent

A point $p\in C\cap D_i$ has local coordinates $(0,y_p)$ in $U_i$,  where $y_p\in\bar{\kp}$ is a root of $P_i$. Under the hypothesis $(H_2)$, the curve $C$ intersects transversally the boundary of $X$ so that there exists a unique series $\phi_{p}\in \bar{\kp}[[x]]$ such that $\phi_p(0)=y_p$ and $f_i(x,\phi_{p})\equiv 0$. The restriction $\gamma_p$  of the germ of $C$ at $p$ to an effective toric divisor $D=\sum k_iD_i$ is thus uniquely determined by the truncation at order $k_i$ of the series $\phi_p$. Since the $\phi_p$'s are conjugate when the $y_p$'s run over the roots of an irreducible rational factor $P$ of $P_i$, the Cartier divisor of $D$
$$
\gamma_P:=\sum_{P(y_p)=0} \gamma_p 
$$
is defined and irreducible over $\kp$. It follows that the restriction $\gamma_C$ of $C$ to $D$ admits the irreducible \textit{rational} decomposition   
$$
\gamma_C=\sum_{P\in\mathcal{P}}\gamma_P,
$$
where $\mathcal{P}:=\mathcal{P}_1\cup\cdots\cup\mathcal{P}_r$ is the union of the sets $\mathcal{P}_i$ of the non constant monic irreducible rational factors of the $i^{th}$ exterior facet polynomial of $f$.

\subsection{Computing the reduced echelon basis. Proof of Theorem $2$.}

We now give a way to compute the reduced echelon basis of $V(D)\subset V$ when
$D=2\Div_{\infty}(f)$. So we need criterions for lifting Cartier divisors from $D$ to $X$. In \cite{W:gnus}, the author obtain such conditions that involves the algebraic coefficients of the series $\phi_p$'s. From an effective point of view, we rather follow \cite{CL:gnus} and work over the residue field
$$
\kp_P:=\kp[y]/(P(y))
$$
associated to each $P\in \mathcal{P}$. Let $y_P$ be the residual class of $y$. A series $\B\in\kp_P[[x]]$ can be uniquely written
$$
\B=\sum_{j=0}^{l_p-1} b_j y_P^j
$$
where $l_P:=\deg(P)$ and $b_j:=\coef(\B,y_P^j) \in\kp[[x]]$. For convenience, we denote by $\coef_k(\B,y_P^j)\in\kp$ the coefficient of $x^k$ in $b_j$. 
\vskip2mm
\noindent

Suppose that $P\in \mathcal{P}_i$. Under the hypothesis $(H_2)$, there exists a unique power series $\phi_P\in\kp_P[[x]]$ such that $f_i(x,\phi_P)\equiv 0$ and $\phi_P(0)=y_P$. So $\phi_P$ is the conjugate class of the series $\phi_p$ associated to the roots of $P$. Note that $\phi_P$ is invertible. For all integer $k\ne 0$, we define
$$
\B^k(\phi_P):=\frac{\phi_P^{k}}{k}\in \kp_P[[x]],
$$
and we define $\B^0(\phi_P):=\log(\phi_P)$ to be the unique primitive of $\phi_P'/\phi_P$ which vanishes at $0$ . For all $m\in\zp^2$, we define
$$
a_{Pm}:=\sum_{j=0}^{l_P-1} \Tr^j(P)\coef_{-\langle m,\eta_{i}\rangle}(\B^{\langle m,\eta_{i+1}\rangle}(\phi_P),y_P^j),
$$
where $i$ is chosen so that $P\in\mathcal{P}_i$ and where $\Tr^j(P)$ designs the sum of the $j^{th}$ power of the roots of $P$. So $a_{Pm}\in\kp$.
Theorem $3$ follows from the following
\vskip3mm
\noindent
\begin{prop}
Let $M$ denote the set of interior lattice points of $2N_f$ and let $A$ denote the matrix $(a_{Pm})_{P\in \mathcal{P}, \,m\in M}$. There is equality
$V(D)=\ker(A)$.
\end{prop}
\vskip3mm
\noindent

\begin{proof}
By the algebraic osculation Theorem $1$ in \cite{W:gnus}, we know that there exists a pairing 
$$
\langle \cdot,\cdot \rangle_D : \Divisor(D)\otimes\cp\times H^0(X,\Omega_{X}^2(D))\rightarrow  \cp
$$
such that $\gamma \in \Divisor(D)$ extends to $X$ if and only if $\langle \gamma,\cdot \rangle_D\equiv 0$. Since $D=2\Div_{\infty}(f)$, we have 
$$
H^0(X,\Omega_{X}^2(D))=\bigoplus_{m\in M} \cp \,\psi_m,\quad \psi_m:=t^m\frac{dt_1\land dt_2}{t_1t_2}
$$
and the explicit formula of Proposition $1$ in \cite{W:gnus} gives equality
$$
\langle \gamma_p,\psi_m \rangle_D=\coef(\B^{\langle m,\eta_{i+1}\rangle}(\phi_{p}),x^{-\langle m,\eta_{i}\rangle})
$$
for all $p\in C\cap D_i$. Note that $-\langle m,\eta_{i}\rangle< 2d_i$ by hypothesis, so that the previous expression only depends on $\phi_p$ modulo $(x^{2d_i})$. 
\vskip2mm
\noindent

Suppose that $y_p$ is a root of $P\in \mathcal{P}_i$. Since $\phi_P$ is the conjugate class of $\phi_p$, there exists for all $k\in\zp$ a unique polynomial $\R^k\in \kp[x][y]$ with degree $<2d_i$ in $x$ and degree $<l_P$ in $y$ such that 
$$
\B^k(\phi_P)\equiv \R^k(y_P)\,\,{\rm mod}\,(x^{2d_i})\quad {\rm and}\quad \B^k(\phi_p)\equiv \R^k(y_p)\,\,{\rm mod}\,(x^{2d_i}).
$$
Hence, we have congruence relations
\begin{eqnarray*}
\sum_{P(y_p)=0} \B^k(\phi_p) &\equiv &\sum_{P(y_p)=0} \sum_{j=0}^{l_P-1} \coef(\R^k,y^{j})y_p^{j}\equiv\sum_{j=0}^{l_P-1} \Tr^j(P)\coef(\B^k(\phi_P),y_P^{j})
\end{eqnarray*}
modulo $(x^{2d_i})$ and the relation
$$
\langle \gamma_P,\psi_m \rangle_D=\sum_{P(y_p)=0}\langle \gamma_p,\psi_m \rangle_D=a_{Pm}
$$
follows. So $\gamma\in V_{\zp}$ extends to a Cartier divisor on $X$ if and only if $\gamma A=0$. It follows that $V(D)=\ker(A)$.
\end{proof}

\vskip0mm
\noindent

Let us look at the algorithmic complexity underlying Proposition $1$. 
Following \cite{GG:gnus}, we use notation $\wt{\mathcal{O}}$ for soft $\mathcal{O}$,  and we let $2<\omega< 2.34$ be the matrix multiplication exponent. We recall that the complexity for multiplying two polynomials of degree $d$ belongs to $\wt{\mathcal{O}}(d)$ (Sch\"{o}nhage and Strassen algorithm,  \cite{GG:gnus}). 
\vskip2mm
\noindent
\begin{coro}
Suppose given the rational factorization of the exterior facet polynomials of $f$. We can build the matrix $A$ with 
$\wt{\mathcal{O}}(\Vol(N_f)^2)$ arithmetic operations in $\kp$ and then compute the reduced echelon basis of $V(D)$ with 
$\mathcal{O}(\Vol(N_f)\Card(\mathcal{P})^{\omega-1})$ arithmetic operations in $\kp$. 
\end{coro}
\vskip2mm
\noindent

\begin{proof}
We divide the proof in three steps. 
\vskip2mm
\noindent
{\it Step 1. Computing the $\phi_P$'s.} We use the classical Newton iteration algorithm \cite{GG:gnus}. In order to estimate the cost in our toric setting, we need the following 
\vskip1mm
\noindent
\begin{lemm}
Let $P\in \mathcal{P}_i$, $\phi\in\kp_P[[x]]$ and let $f_i\in\kp[x,y]$ as defined in $(13)$. For any $k\in\np$, we can evaluate 
$f_i(x,\phi)\in \kp_P[[x]]$ modulo $(x^k)$ with $\wt{\mathcal{O}}(k\Vol(N_f))$ arithmetic operations in $\kp_P$.
\end{lemm}

\begin{proof}
Since $f_i$ is a sum of $\mathcal{O}(\Vol(N_f))$ monomials, 
we can evaluate $f_i(x,\cdot)$ at $\phi$ modulo $(x^k)$ by evaluating each of the involved monomials with $\mathcal{O}(\Vol(N_f)\log(n_i))$ operations in $\kp_P[[x]]/(x^k)$, where $n_i$ is the total degree of $f_i$ in $y$. All what we need to show is that $n_i$ is not ``too big''. By $(13)$, we have
\begin{eqnarray*}
n_i:=\deg_y(f_i)=\max_{m\in N_f} \langle m,\eta_{i+1}\rangle-\min_{m\in N_f} \langle m,\eta_{i+1}\rangle,
\end{eqnarray*}
Let $m\in N_f$. By $(H_1)$, the polytope $\Conv\{(0,0),(1,0),(0,1),m\}$ is contained in $N_f$. Since it has euclidean volume $(m_1+m_2)/2$, we deduce $\Vert m \Vert\in\mathcal{O}(\Vol(N_f))$ for all $m\in N_f$, where $\Vert m \Vert:=|m_1|+|m_2|$. There remains to estimate $\Vert \eta_{i+1}\Vert$. As before, we check easily that $\Vert\eta\Vert\in\mathcal{O}(\Vol(N_f))$ for all inward primitive normal vectors of the \textit{one dimensional} faces of $N_f$. Let $j> i$ be the first index for which the $j^{th}$-exterior face of $N_f$ has dimension one. So we can write $\eta_{i+1}=a\eta_i+b\eta_{j}$ for some positive rational numbers $a$, $b$. We have relations
$$
b\det (\eta_i,\eta_j)=\det(\eta_i,\eta_{i+1})=1\quad {\rm and} \quad a\det(\eta_i,\eta_j)=\det(\eta_{i+1},\eta_j)<\det (\eta_i,\eta_j),
$$ 
last inequality using that $\Sigma$ is regular and refines $\Sigma_f$ (see \cite{Cox:gnus}). So $a, b\le 1$. Since  $\Vert \eta_i \Vert \in\mathcal{O}(\Vol(N_f))$ and $\Vert \eta_j \Vert \in\mathcal{O}(\Vol(N_f))$, it follows that $\Vert\eta_{i+1}\Vert\in\mathcal{O}(\Vol(N_f))$. Since
$
|\langle m,\eta_{i+1}\rangle|\in\mathcal{O}(\Vert m\Vert\Vert \eta_{i+1}\Vert),
$
we deduce $n_i\in\mathcal{O}(\Vol(N_f)^2)$. 

Hence, the cost for the evaluation step is $\mathcal{O}(\Vol(N_f)\log(\Vol(N_f))$ operations in $\kp_P[[x]]/(x^k)$, or $\wt{\mathcal{O}}(k\Vol(N_f))$ operations in $\kp_P$ (Corollary $9.7$, \cite{GG:gnus}). 
\end{proof}
 
By using the fast modular Newton iteration Algorithm $2$, \cite{CL:gnus} and by replacing the given evaluation cost by that induced by Lemma $9$, we deduce that we can compute the series $\phi_P$ with precision $2d_i$ with
$\wt{\mathcal{O}}(d_i\Vol(N_f))$ operations in $\kp_P$. Each operation in $\kp_P$ takes $\wt{\mathcal{O}}(l_P)$ operations in $\kp$ and we have $\sum_{P\in\mathcal{P}_i}l_P=l_i$. Since $2C$ and $D$ have the same Picard class, we deduce
$$
\sum_{i=1}^r 2d_il_i=\sum_{i=1}^r 2d_i\deg(C\cdot D_i) = \deg(C\cdot D) = 2\deg(C\cdot C)=4\Vol(N_f),
$$
the last equality using basic toric intersection theory (see \cite{F:gnus} for instance). It follows that computing all the $\phi_P$'s up to the precision imposed by $D$ has complexity $\wt{\mathcal{O}}(\Vol(N_f)^2)$. 
\vskip3mm
\noindent
{\it Step 2. Building the matrix $A$.}
Let $P\in\mathcal{P}_i$. Computing the series 
$\B^{\langle m,\eta_{i+1}\rangle}(\phi_P)$ with precision $2d_i$
for all  $m\in \interieur(2N_f)\cap\zp^2$ requires at most one inversion in $\kp_P[[x]]/(x^{2d_i})$ and $\mathcal{O}(\Vol(N_f))$ evaluations of monomials in $\kp_P[[x]]/(x^{2d_i})[y]$ of degrees bounded by $n_i=\mathcal{O}(\Vol(f)^2)$ (and possibly  the computation of a primitive of $\phi_P'/\phi_P$ modulo $(x^{2d_i})$). Each operation takes $\wt{\mathcal{O}}(l_Pd_i)$ operations in $\kp$, giving a total number of  $\wt{\mathcal{O}}(\Vol(N_f)l_Pd_i)$ operations in $\kp$. 
Summing up over all $P\in\mathcal{P}$, we obtain a total number of
$\wt{\mathcal{O}}(\Vol(N_f)^2))$ operations in $\kp$ for computing all the $\B^k(\phi_P)$ involved in the definition of $A$. Then building the matrix $A$ has a negligeable cost.
\vskip3mm
\noindent
{\it Step 3. Computing the reduced echelon basis of $V(D)$.} Since $V(D)$ is determined by $\mathcal{O}(\Vol(N_f)$ equations and $\Card(\mathcal{P})$ unknowns, we can compute its reduced echelon basis with  $\mathcal{O}(\Vol(N_f)\Card(\mathcal{P})^{\omega-1})$ operations in $\kp$ (\cite{Sto:gnus}, Theorem $2.10$). 
\end{proof}

\subsection{Factors computation.} We want now to compute the rational factors of $f$. In all of this section, we fix $\gamma$ in the reduced echelon basis of $V(D)$ and we denote by $q$ the corresponding factor $f$. We first compute the Newton polytope of $q$. By a Theorem of Ostrovski \cite{O:gnus}, $N_{q}$ is a Minkovski summand of $N_f$ so that it is enough to compute the integers 
$$
e_i:=-min_{m\in N_{q}} \langle m,\eta_i\rangle,\quad i=0,\ldots,r+1.
$$
Suppose that $\gamma=\sum_{P\in\mathcal{P}}\mu_P\gamma_P$. We define
$$
l_i(\gamma):=\sum_{P\in\mathcal{P}_i} \mu_{P}\deg(P)
$$
for all $i=1,\ldots,r$. We obtain the following
\vskip3mm
\noindent
\begin{prop}
We have $e_0=e_{r+1}=0$ and the integers $e_1,\ldots,e_r$ are the unique solutions of the affine system 
$$
\sum_{i=1}^r e_{i}a_{ij}=l_j(\gamma),\quad j=1,\ldots,r,
$$
with $a_{i,i+1}:=1$, $a_{ii}:=\det(\eta_{i-1},\eta_{i+1})$ and $a_{ij}:=0$ for $j\ne i,i+1$.
\end{prop}
\vskip3mm
\noindent
\begin{proof}
Since $N_{q}$ is a Minkowski summand of $N_f$, the hypothesis $(H_1)$ forces $0\in N_{q}$, so that $e_{0}=e_{r+1}=0$, and $e_{i}\ge 0$ otherwise. By Subsection $3.1$, we know that
$$
\Div_{\infty}(q)=e_1D_1+\cdots+e_rD_r
$$
while the component $C'=\Div_0(q)$ of $C$ satisfies
\begin{eqnarray}
\deg(C'\cdot D_j)=\deg(\gamma_{|D_j})=l_j(\gamma), \quad j=1,\ldots,r.
\end{eqnarray}
Now, the Chow group of the smooth affine plane completion $X$ is $\zp$-free generated by the classes of $D_1,\ldots,D_r$ and numerical and rational equivalence coincide on a smooth toric variety \cite{F:gnus}. By duality, it follows that $(14)$ uniquely determines the class of $C'$, that is the class of $\Div_{\infty}(q)$, that is the integers $e_{1},\ldots,e_{r}$. We have equality
$$
\deg(C'\cdot D_j)=\deg(\Div_{\infty}(q)\cdot D_j)=\sum_{i=1}^r e_{i}\deg(D_i\cdot D_j)
$$
and it's well known that $\deg(D_i\cdot D_j)=a_{ij}$ (see \cite{F:gnus} for instance). Proposition $2$ follows. Note that $\det(a_{ij})=\pm 1$.
\end{proof}

\vskip3mm
\noindent

We can now consider $q(t)=\sum_{m\in N_q\cap \zp^2}c_m t^m$ as a vector indexed by the lattice points of $N_q$. For all $m\in N_q$ and all $P\in\mathcal{P}$, we define

\vskip3mm
\noindent
$$
r_{Pm}\in\kp_P[[x]]/(x^{e_i+1}),\qquad r_{Pm}:=x^{\langle m,\eta_i\rangle+e_i}\phi_P^{\langle m,\eta_{i+1}\rangle}\,\, {\rm mod}\, (x^{e_i+1}),
$$
\vskip3mm
\noindent
where $i$ is chosen so that $P\in\mathcal{P}_i$. This definition makes sense since $\phi_P$ is invertible and $\langle m,\eta_i\rangle+e_i\ge 0$ for all $m\in N_q$. We obtain the following 

\vskip3mm
\noindent
\begin{prop} 
Under  normalization $q(0)=1$, the coefficients vector $(c_m)_{m\in N_q\cap \zp^2}$ of $q$
is the unique rational solution of
$$
c_0=1\quad and \quad \sum_{m\in N_q\cap \zp^2} c_m r_{Pm}=0 \in\kp_P[[x]]/(x^{e_i+1})\quad \forall P\in\mathcal{P},\,\,\mu_P\ne 0.
$$
The induced system $(S_{\gamma})$ over $\kp$ contains $\mathcal{O}(\Vol(N_q))$ affine equations.
\end{prop}
\vskip3mm
\noindent

\begin{proof}
Let $h\in\kp[t_1,t_2]$ with Newton polytope contained in $N_q$. So $h$ and $q$ define global sections of the line bundle $\mathcal{O}_X(E)$, where $E:=\sum_{i=1}^r e_iD_i$. In particular $h=\lambda q$ for $\lambda\in\cp$ if and only if the two Cartier divisors $H:=\Div(h)-E$ and  $C':=\Div(q)-E$ are equal. Since the restriction $$H^0(X,\mathcal{O}_X(E))\rightarrow H^0(X,\mathcal{O}_{E+\partial X}(E))$$ is injective and $H$ and $C'$ are rationally equivalent to $E$, we deduce that 
\begin{eqnarray*}
C'=H\,&\iff & C'\cap (E+\partial X)= H\cap (E+\partial X)\\
 &\iff & C'\cap (E+\partial X)\subset H,
\end{eqnarray*}
with intersection and inclusion taken scheme theoretically. Since $N_q\subset N_f$, we have
$E\le \Div_{\infty}(f)$, so that $E+\partial X\le D$ and
$$
C'\cap (E+\partial X)=\gamma\cap (E+\partial X)=\bigcup_{\mu_P\ne 0} \gamma_P\cap (E+\partial X).
$$  
Let $h_i=0$ be the local equation of $H$ in the chart $U_i$. For $P\in\mathcal{P}_i$, we know that $\gamma_P\cap (E+\partial X) =\gamma_P\cap (e_i+1)D_i$ is contained in $U_i$. Hence, we have equivalence
$$
\gamma_P\cap (e_i+1)D_i\subset H \iff h_i(x,\phi_P(x))\equiv 0\quad {\rm mod}\, (x^{e_i+1}).
$$
If $h(t)=\sum_{m\in N_q \cap \zp^2} u_m t^m$, we can chose (following  $(13)$)
\begin{eqnarray*}
h_i(x,y):=\sum_{m\in N_q \cap \zp^2} u_m x^{\langle m,\eta_i\rangle+e_i}y^{\langle m,\eta_{i+1}\rangle+e_{i+1}}.
\end{eqnarray*}
Since  $\phi_P$ is invertible, we obtain  by linearity that $\gamma_P\cap (E+\partial X)\subset H$ is equivalent to that equality $\sum_m u_mr_{Pm}=0$ holds in $\kp_P[[x]]/(x^{e_i+1})$. It follows that $q$ is the unique solution of the announced system of equations. 

A linear equation in $\kp_P[[x]]/(x^{e_i+1})$ being equivalent to a system of $l_P(e_i+1)$ equations over $\kp$, the total number of equations of the induced system $(S_{\gamma})$ over $\kp$ is
\begin{eqnarray*}
1+\sum_{i=1}^r\sum_{P\in\mathcal{P}_i}  (e_i+1)\mu_P l_P&=&1+\sum_{i=1}^r (e_i+1)\deg(C'\cdot D_i) \\
&=& \deg(C'\cdot E)+ \deg(C'\cdot \partial X)+1\in \mathcal{O}(\Vol(N_q)).
\end{eqnarray*}
\end{proof}
\vskip0mm
\noindent
We deduce the following
\vskip2mm
\noindent
\begin{coro}
Suppose given the reduced echelon basis of $V(D)$. We can compute all the irreducible rational factors of $f$ with at most 
$
\mathcal{O}(\Vol(N_f)^{\omega})
$
arithmetic operations in $\kp$.
\end{coro}
\vskip2mm
\noindent
\begin{proof}
Let $\gamma$ belongs to the reduced echelon basis of $V(D)$. Once the affine system $(S_{\gamma})$ is built, Proposition $3$ allows us to compute the corresponding factor $q$ of $f$ by solving a system of $\mathcal{O}(\Vol(N_q))$ affine equations over $\kp$ and $\mathcal{O}(\Vol(N_q))$ unknowns. This requires
$
\mathcal{O}(\Vol(N_q)^{\omega})
$
arithmetic operations in $\kp$ (Theorem $2.10$ in \cite{Sto:gnus}). There remains to build the system. Computing $N_q$ using Proposition $2$ has a negligeable cost. Let $P\in\mathcal{P}_i$. We need to compute $\phi_P$ with precision $x^{e_i+1}$. Since $e_i \le d_i$ and $d_i>0$, we have $e_i+1\le 2d_i$ so that $\phi_P$ has already been computed with a sufficient precision. In the same way, computing the involved powers $\phi_P^{\langle m,\eta_{i+1}\rangle}$, $m\in N_q$ from the already computed powers $\phi_P^{\langle m,\eta_{i+1}\rangle}$, $m\in\interieur (2N_f\cap\zp^2)$ has a negligeable cost too. If $f=q_1\cdots q_s$ is the rational factorization of $f$, we have inequality
$$
\Vol(N_{q_1})+\cdots +\Vol(N_{q_s})\le \Vol(N_{q_1}+\cdots +N_{q_s})= \Vol(N_f)
$$
and we finally need at most
$$
\mathcal{O}(\Vol(N_{q_1})^{\omega}+\cdots +\Vol(N_{q_s})^{\omega})\subset \mathcal{O}(\Vol(N_f)^{\omega})
$$
arithmetic operations in $\kp$ for computing the rational factorization of $f$ from the reduced echelon basis. 
\end{proof}

\vskip2mm
\noindent
Let us remark that the system $(S_{\gamma})$ has a particular sparse structure. For instance, it contains the subsystems of type Vandermond 
$$
\sum_{m\in N_q^{(i)}\cap \zp^2}c_m  y_P^{\langle m,\eta_{i+1}\rangle}=0,\quad P\in\mathcal{P}_i,\,\mu_P\ne 0
$$
that determines (up to multiplication by some constant) the $i^{th}$ exterior facet polynomial $\prod_{P\in\mathcal{P}_i} P^{\mu_P}$ of $q$. We might hope that in practice, the resolution of $S_{\gamma}$ is relatively fast.  

\vskip3mm
\noindent

\subsection{A toric factorization algorithm. Proof of Theorem $1$.}

By combining all previous results, we deduce the following
\vskip3mm
\noindent
{\bf Toric Factorization Algorithm (TFA).}

\vskip2mm
\noindent
{\it Input:} $f\in \kp[t_1,t_2]$ satisfying hypothesis $(H_1)$ and $(H_2)$.
\vskip1mm
\noindent
{\it Output:} The irreducible factorization $f=q_1\cdots q_s$ of $f$ over $\kp$.
\vskip3mm
\noindent
{\it Step 0.} Compute a regular fan $\Sigma$ which refines $\Sigma_f$.
\vskip1mm
\noindent
{\it Step 1.} Compute the set $\mathcal{P}$ of the irreducible rational factors of the exterior facet polynomials of $f$.
\vskip1mm
\noindent
{\it Step 2.} For $i=1,\ldots,r$ and $P\in\mathcal{P}_i$, compute the series $\phi_P\in\kp[[x_i]]$ with precision $x_i^{2d_i}$ using  modular Newton iteration.
\vskip1mm
\noindent
{\it Step 3.} Compute the suitable powers of the $\phi_P$'s in order to build the matrix $A$ of Proposition $1$.
\vskip1mm
\noindent
{\it Step 4.} Compute the reduced echelon basis of $V(D)$.
\vskip1mm
\noindent 
{\it Step 5.} Compute the  rational factors of $f$ by using Propositions $2$ and $3$.
\vskip4mm
\noindent
Theorem $1$ follows immediately from the following
\vskip2mm
\noindent
\begin{prop} The algorithm TFA is correct. It requires to factorize the exterior facet polynomials and to perform at most 
$\mathcal{O}(\Vol(N_f)^{\omega})$ 
arithmetic operations in $\kp$.
\end{prop}
\vskip2mm
\noindent
\begin{proof} 
The correctness of the algorithm is a consequence of Theorem $2$ and Propositions  $1$, $2$, $3$. The desingularization of the fan $\Sigma$ can be obtained by computing some Hirzebruch continued fractions (see \cite{Cox:gnus}) and has a negligeable cost.  We consider rational univariate factorization as a black-box of our algorithm. See for instance \cite{GG:gnus}, \cite{Nov:gnus} and \cite{Bel:gnus} for recent advances in that direction. Finally, the cost of steps $2,3,4,5$ follows from Corollaries $1$ and $2$. 
\end{proof}

\vskip2mm
\noindent
\section{Comparison with related results. Improvements}

In Subsection $4.1$, we compare the algorithm TFA with the most related dense algorithms. In Subsection $4.2$, we discuss the relation with the toric algorithm developed in \cite{W:gnus} by the author. In particular, we obtain a sufficient criterion for using a smaller lifting precision.

\subsection{Comparison with dense algorithms}
 
We compare our method with the lifting and recombination scheme proposed by Lecerf  \cite{L:gnus}, \cite{L2:gnus} and Ch\`eze-Lecerf \cite{CL:gnus} for dense polynomials and we discuss some possible improvements for each step of the algorithm TFA.
\vskip2mm
\noindent 
{\it About Step $0$.} Since $C\cap D_i=0$ for all rays $\rho_i\in\Sigma\setminus \Sigma_f$, we need not to compute all the fan $\Sigma$. Namely, we check easily that it's enough to compute the succesive rays $\rho_{i+1}\in\Sigma$ of the rays $\rho_i\in \Sigma_f$. 

\vskip2mm
\noindent 
{\it About Step $1$.} In most cases, the cost of the univariate factorization step dominates  the complexity of the algorithm TFA. It might be interesting to avoid some of the facet factorization by chosing $D$ with support strictly contained in $|\partial X|$. For instance, if $f$ has bidegree $(d_1,d_2)$, then $X=\pp^1\times \pp^1$ and we might hope to recover the decomposition of $C$ from its restriction to $D=(d_1+1)\pp^1$ since the corresponding restriction $H^0(X,\mathcal{O}_X(C))\rightarrow H^0(X,\mathcal{O}_D(C))$ is injective. This turns out to be the case : in \cite{L:gnus}, G. Lecerf factorizes bidegree polynomials by using only one facet factorization with a sharp precision. In general, $D$ has to obey to the vanishing cohomological properties used in the proof of Theorem $2$, which are closely connected with the geometry of $N_f$. 
 
\vskip2mm
\noindent 
{\it About Step $2$.} In the dense absolute case treated in \cite{CL:gnus}, G. Lecerf and G. Ch\`eze compute the analoguous series by introducing the Paterson-Stockmeyer evaluation scheme in the Newton iteration process. We can adapt such a method to our situation by replacing the input polynomial of Algorithm $1$, \cite{CL:gnus} by a polynomial with degree bounded by $\Vol(N_f)$. In such a way, the  complexity $\mathcal{O}(\Vol(N_f)^{2})$ of step $2$ decreases to $\wt{\mathcal{O}}(\Vol(N_f)^{(\omega+1)/2})$.
\vskip2mm
\noindent 
{\it About Step $3$.} In  \cite{L:gnus}, G. Lecerf builds an analoguous linear system by using fast modular euclidean divison rather than by computing the $\phi_P$'s powers. Both approaches give rise to equivalent linear systems (see \cite{BLSSW:gnus}, Section $2.3$), but the  divison method permits to build the underlying matrix faster. We might hope that in the toric case, it is possible too to introduce an equivalent matrix that can be built using modular division.
\vskip2mm
\noindent 
{\it About Step $4$.} In \cite{L:gnus} and \cite{L2:gnus}, the linear system resolution has complexity $\mathcal{O}(d^{\omega+1})$ with $d$ the total degree of $f$. It's easy to check that the sum of the lattice lengths of the exterior facets of $N_f$ is bounded by $d$, with equality if and only if $\Sigma_f$ is regular. It follows that $\Card(\mathcal{P})\le d$. Since $\mathcal{O}(\Vol(N_f))\subset \mathcal{O}(d^2)$, the reduced echelon basis computation is faster using the toric approach (much faster in the most case).

\vskip2mm
\noindent 
{\it About Step $5$.}
Since we recover the factors of $f$ by solving affine systems, step $5$ has a relatively high cost in the algorithm TFA. If $f$ is a dense polynomial, the task is much simpler and we can recover fastly the global factors of $f$ from the local ones by using modular multiplications of (see \cite {L:gnus}), or by using a partial fraction decomposition method (see \cite{CL:gnus}, \cite{Gao:gnus}). This permits a softly $d^3$ complexity for the factors computation, in general much faster than our approach. We  might hope to adapt these methods to the toric case.

\subsection{About the lifting precision}

In \cite{L2:gnus}, G. Lecerf gives an example in the dense case that shows that the precision $D:=2\Div_{\infty}(f)$ is sharp in Theorem $2$.  
On an other hand, the author obtain in \cite{W:gnus} a toric factorization algorithm running with precision $E:=\Div_{\infty}(f)+\partial X$, but with exponential complexity in most cases. We explain here the relation with our algorithm and  we give an explicit sufficient criterion for using the precision $E<D$. 
\vskip2mm
\noindent 

Let $\gamma=\sum_{P\in\mathcal{P}} \mu_P \gamma_P\in V$. We define the rational numbers
$$
l_i(\gamma):=\sum_{P\in\mathcal{P}}\mu_P \deg(P), \quad i=1,\ldots, r
$$
and we introduce the following convex subset of $V$
$$
\Delta_C=\Big\{\gamma\in V,\quad \sum_{i=1}^r \langle e_k,\eta_i \rangle l_i\le \sum_{i=1}^r \langle e_k,\eta_i \rangle l_i(\gamma) \le 0, \quad  k=1,2\Big\},
$$
where $(e_1,e_2)$ is the canonical basis. We have the following 

\vskip2mm
\noindent
\begin{prop}
The finite set $V(E)\cap \{0,1\}^{\mathcal{P}} \cap \Delta_C$ is a system of generators of the vector subspace $V(D)\subset V(E)$.
\end{prop}
\vskip2mm
\noindent
\begin{proof}
Since $V(D)$ admits a basis with coordinates in $\{0,1\}$, it's enough to show that $V(D)\cap\{0,1\}^{\mathcal{P}}=V(E)\cap \{0,1\}^{|\Gamma|} \cap \Delta_C$. Let $\gamma\in V(D)\cap \{0,1\}^{\mathcal{P}}$. So $\gamma$ is restriction to $D$ of a rational component $C'\in \Divisor(X)$ of $C$ (Theorem $2$). In particular, both divisors $C'$ and $C-C'$ are numerically effective, which is equivalent to that
\begin{eqnarray}
0\le \deg(C'\cdot D_i)\le \deg(C\cdot D_i),\quad i=0,\ldots,r+1.
\end{eqnarray}
For $i=1,\ldots,r$, we have equalities 
$\deg(C'\cdot D_i)=l_i(\gamma)$ and $\deg(C\cdot D_i)=l_i$. On an other hand, for any $m\in\zp^2$, we have
$$
\sum_{i=0}^{r+1} \langle m,\eta_i \rangle \deg(C'\cdot D_i)=\sum_{i=0}^{r+1} \langle m,\eta_i \rangle \deg(C\cdot D_i)=0,
$$
since the divisor $\sum_{i=0}^{r+1} \langle m,\eta_i \rangle D_i=\Div(t^m)$ is principal. Letting $m=e_1$ and using that $(\eta_{r+1},\eta_0)=(e_1,e_2)$, we deduce that
$$
\deg(C'\cdot D_0)=-\sum_{i=1}^r \langle e_1,\eta_i \rangle l_i(\gamma),\quad \deg(C\cdot D_{0})=-\sum_{i=1}^r \langle e_1,\eta_i \rangle l_i,
$$
and the same reasonning with $m=e_2$ gives 
$$
\deg(C'\cdot D_{r+1})=-\sum_{i=1}^r \langle e_2,\eta_i \rangle l_i(\gamma),\quad \deg(C\cdot D_{r+1})=-\sum_{i=1}^r \langle e_2,\eta_i \rangle l_i.
$$
Combined with $(15)$, we deduce that $\gamma\in \Delta_C$, giving an inclusion $V(D)\cap\{0,1\}^{\mathcal{P}}\subset V(E)\cap \{0,1\}^{\mathcal{P}} \cap \Delta_C$. 

Let us show the opposite inclusion. If $\gamma\in V(E)\cap \{0,1\}^{\mathcal{P}}\cap \Delta_C$, it lifts to some divisor $C'\in\Divisor(X)$. By hypothesis, we have $0\le \gamma\le \gamma_C$, giving obvious inequalities
$$
0\le \deg(C'\cdot D_i)=l_i(\gamma)\le l_i= \deg(C\cdot D_i)
$$
for all $i=1,\ldots ,r$. Since $\gamma\in \Delta_C$, we deduce from the previous discussion that $(15)$ holds.  Since being nef is equivalent to being globally generated on a toric variety,it follows that both $\mathcal{O}_X(C')$ and $\mathcal{O}_X(C-C')$ are globally generated. By \cite{W:gnus}, proof of Theorem $2$, this gives rise to supplementary vanishing cohomology properties which ensure that $C'$ can be chosen to be an absolute component of $C$. Since $\gamma$ is defined over $\kp$, such a component is rational by Lemma $8$. Thus $\gamma$ is a $\{0,1\}$-linear combination of the $\gamma_j$'s, that is $\gamma\in V(D)\cap \{0,1\}^{\mathcal{P}}$. 
\end{proof}

Proposition $5$ admits the following useful corollary, which is the toric version of \cite{BLSSW:gnus}, Proposition $4$.
\vskip2mm
\noindent
\begin{coro}
If all the exterior facets of $N_f$  have an inward primitive normal vector with negative coordinates, then $V(E)=V(D)$ if and only if each vector of the reduced echelon basis of $V(E)$ lies in $\{0,1\}^{\mathcal{P}}$.
\end{coro}
\vskip1mm
\noindent
\begin{proof}
If the reduced echelon basis of $V(E)$ lies in $\{0,1\}^{\mathcal{P}}$, we have $l_i(\gamma)\le l_i$ for all $\gamma$ in that basis. By hypothesis, we have $\langle e_k,\eta_i \rangle\le 0$ for all $i=1,\ldots,r$, $k=1,2$ and it follows that $\gamma\in V(E)\cap \{0,1\}^{\mathcal{P}} \cap \Delta_C$. The equality $V(E)=V(D)$ follows from Proposition $5$. The other implication is trivial. 
\end{proof}

Proposition $5$ gives an efficient way to compute the reduced echelon basis of $V(D)$ from the finite set $V(E)\cap \{0,1\}^{\mathcal{P}}$. Roughly speaking, the underlying algorithm is that developed in \cite{W:gnus}. It has the advantage to use a smaller lifting precision, but in return, it looks for ``good'' partitions of $\{0,1\}^{\mathcal{P}}$ and can have an exponential complexity. We don't know what is the probability for that equality $V(E)=V(D)$ holds.

\section{Conclusion}
We propose a new lifting and recombination algorithm for rational bivariate factorization that takes advantage of the geometry of the Newton polytope. For polynomials that are sparse enough, our complexity is competitive with that of the actual fastest algorithms developed for dense polynomials. We might hope to improve the complexity with a carefull application of the standard  modular algorithms  in the toric setting.
\vskip2mm
\noindent
{\bf Acknowledgments.} We thank Jos\'{e} Ignacio Burgos and Martin Sombra for their careful reading and helpful comments.

\end{document}